\documentclass[12pt,english]{article}
\usepackage[T1]{fontenc}
\usepackage[utf8]{inputenc}
\usepackage{geometry}
\usepackage{babel}
\usepackage{mathtools}
\usepackage{amsthm}
\usepackage{amssymb}
\usepackage{scalerel}
\usepackage[unicode=true,pdfusetitle,
 bookmarks=true,bookmarksnumbered=false,bookmarksopen=false,
 breaklinks=false,pdfborder={0 0 1},backref=false,hidelinks]
 {hyperref}

\makeatletter
\numberwithin{equation}{section}
\numberwithin{figure}{section}
\theoremstyle{plain}
\newtheorem{thm}{\protect\theoremname}
\theoremstyle{definition}

\theoremstyle{remark}

\theoremstyle{plain}
\newtheorem{lem}[thm]{\protect\lemmaname}
\theoremstyle{plain}
\newtheorem{prop}[thm]{\protect\propositionname}

\date{}
\usepackage{tikz-cd}
\IfFileExists{libertine.sty}{\usepackage[tt=false]{libertine}}{}
\usepackage[parfill]{parskip}
\def\thm@space@setup{%
  \thm@preskip=10pt plus 2pt minus 2pt
  \thm@postskip=10pt plus 2pt minus 2pt
}
\usepackage{enumitem}
\setlist[itemize]{noitemsep,topsep=5pt}

\usepackage{titlesec}
\usepackage{titletoc}
\titleformat{\section}{\large\bfseries\filleft}{\thesection}{1em}{}[{\titlerule[0.8pt]}]
\titleformat{\part}[block]
  {\normalfont\Large\bfseries\raggedright}
  {}{0pt}{}
\titlespacing*{\part}{0pt}{3.5ex plus 1ex minus .2ex}{1.5ex}

\titlecontents{part}
  [0pt]
  {\addvspace{.9em}\bfseries}
  {}
  {}
  {}
\titlecontents{section}
  [1.5em]
  {\setlength{\parfillskip}{0pt plus 1fil}\contentsmargin{0pt}}
  {\contentslabel{1.75em}}
  {}
  {\titlerule*[.6pc]{.}\hspace*{1.2em}\makebox[0pt][r]{\thecontentspage}}

\newenvironment{ack}{\textit{Acknowledgements.}}{}
\renewcommand\labelenumi{(\roman{enumi})}
\renewcommand\theenumi\labelenumi

\makeatletter
\DeclareRobustCommand{\bigboxplus@}{%
  \mathop{\vphantom{\sum@}\scaleobj{0.9}{\scalerel*{\boxplus}{\sum}}}%
}
\newcommand{\bigboxplus}{\DOTSB\bigboxplus@\slimits@}
\DeclareRobustCommand{\bigboxtimes@}{%
  \mathop{\vphantom{\sum@}\scaleobj{0.9}{\scalerel*{\boxtimes}{\sum}}}%
}
\newcommand{\bigboxtimes}{\DOTSB\bigboxtimes@\slimits@}
\makeatother

\DeclareMathOperator{\Supp}{Supp}
\DeclareMathOperator{\codim}{codim}
\DeclareMathOperator{\ch}{char}

\DeclareMathOperator{\divisor}{div}

\DeclareMathOperator{\Sym}{Sym}

\DeclareMathOperator{\mult}{mult}

\DeclareMathOperator{\res}{res}

\DeclareMathOperator{\cdeg}{cdeg}

\DeclareMathOperator{\Sing}{Sing}

\DeclareMathOperator{\ord}{ord}

\makeatother

\providecommand{\definitionname}{Definition}
\providecommand{\lemmaname}{Lemma}
\providecommand{\propositionname}{Proposition}
\providecommand{\theoremname}{Theorem}

\global\long\def\C{\mathbb{C}}%
\global\long\def\F{\mathbb{F}}%
\global\long\def\P{\mathbb{P}}%
\global\long\def\O{\mathcal{O}}%
\global\long\def\Z{\mathbb{Z}}%
\global\long\def\ep{\varepsilon}%

\title{The asymptotic in Waring's problem over function fields beyond twice the degree}
\author{Matthew Hase-Liu}

\begin{document}

\maketitle

\begin{abstract}
We prove the expected asymptotic in Waring's problem over $\F_q[T]$, with a power-saving error, whenever $n>2d$, the characteristic is greater than $(d-1)^2$, and $q$ satisfies an explicit lower bound. This range is sharp in general for the expected asymptotic uniformly in the target polynomial.

Our main new input is an aggregate minor arc estimate: we count functionals according to the codimension of their associated singular loci and use intersection theory to bound the degrees of the resulting parameter spaces. In particular, if $n\ge (2+\ep)d$, the required lower bound on $q$ is polynomial in $d$ of degree $2+4/\ep$.
\end{abstract}
\setcounter{tocdepth}{1}
\tableofcontents{}

\section{Introduction}

Over the integers, the expected asymptotic in Waring's problem was first established by Hardy and Littlewood when the number of variables is sufficiently large. The best known bound for the number of variables required is quadratic in $d$ \cite{BourgainVMVT,WooleyAsymptotic,WooleyNested}. The problem of reducing the number of variables is also closely tied to the main conjecture in Vinogradov's mean value theorem: Wooley proved the cubic case using efficient congruencing \cite{WooleyCubic}, Bourgain--Demeter--Guth proved the cases of degree greater than three using decoupling \cite{BDGVMVT}, and Wooley subsequently gave a proof in all degrees using nested efficient congruencing \cite{WooleyNested}.

The function field Waring problem was studied earlier by Kubota
\cite{KubotaWaring} and Liu and Wooley \cite{LiuWooleyWaring}.
For the asymptotic problem over $\F_q[T]$, Yamagishi obtained
bounds that are quadratic in the degree, as well as linear bounds in
certain special cases \cite{YamagishiWaring}. For cubes, Glas and
Hochfilzer established an asymptotic formula with a power-saving error
term in seven variables whenever $\ch(\F_q)\ne3$ \cite{GlasHochfilzer}. Developing an idea of Pugin \cite{PuginThesis}, Sawin instead treated the minor arc sums as complete exponential sums over $\F_q$ and applied estimates in
terms of their singular loci \cite{SawinWaring}. This reduced the number
of variables from quadratic to linear in the degree; for sufficiently large characteristic and $q$, Sawin's result requires roughly $4d$ variables.

Write $p=\ch(\F_q)$.  
Let $P_j=H^0(\P^1,\O(j\infty))$, which we identify with the
polynomials in $\F_q[T]$ of degree at most $j$.  For $f\in P_{de}$,
write
\[N_{e,n}(f)=\#\{(a_1,\ldots,a_n)\in P_e^n\colon a_1^d+\cdots+a_n^d=f\}.\]
The asymptotic regime is $e\to\infty$ with $q$, $d$, and $n$ fixed.
We use the local factors $\ell_\infty(f)$ and $\ell_\pi(f)$ defined in
\cite[Problem 1.1]{SawinWaring}, writing $\ell_{\infty,e}(f)$ for the
factor at infinity to record its dependence on the degree bound.

The main result of this paper establishes the following expected fixed field estimate.

\begin{thm}\label{thm:main}
Fix integers $d\ge3$ and $n>2d$. Suppose
\[
 p>(d-1)^2 \text{ and } q > (d-1)^{\frac{2n}{n-2d}}\left(\frac{12}{5}d\cdot100^d\right)^{\frac{4}{n-2d}}.
\]
Then there is $\theta=\theta(q,d,n)>0$ such that, uniformly for
$e\ge1$ and $f\in P_{de}$,
\begin{equation}\label{eq:main-asymptotic}
 N_{e,n}(f)=q^{e(n-d)+n-1}\ell_{\infty,e}(f)
 \prod_{\pi\textnormal{ monic irreducible}}\ell_\pi(f)
 +O_{q,d,n}\left(q^{e(n-d)+n-1-\theta e}\right).
\end{equation}
\end{thm}
The following proposition shows that the range $n>2d$ is sharp.
\begin{prop}\label{prop:sharpness}
    Let $d\ge 3$, $p>d$, and $\F_q$ contain an element $\zeta$ with $\zeta^d=-1$. For fixed $q$, as $e\to \infty$, when $n=2d$, the difference between $N_{e,n}(0)$ and the main term in \eqref{eq:main-asymptotic} is not $o(q^{e(n-d)+n-1}).$ 

    Also, for $n=2d-1$, the corresponding error is likewise not $o(q^{e(n-d)+n-1})$ uniformly in $f$.
\end{prop}

The proof is given at the very end in Section \ref{sec:puttingeverythingtogether}. This is even easier to see for $n\le 2d-2$: Put $r=\lfloor n/2\rfloor$. Using the notation of the proposition, for arbitrary $b_1,\ldots,b_r\in P_e$, set $a_{2i-1}=b_i$ and $a_{2i}=\zeta b_i$ for $1\le i\le r$, and, when $n$ is odd, set the final coordinate equal to zero. This gives $q^{r(e+1)}$ solutions with $f=0$, which is larger than $q^{e(n-d)+n-1}$.

This is a direct continuation of Sawin's analysis in
\cite{SawinWaring}.  We use his circle method decomposition and
major arc calculation, his description of functionals by finite
subschemes of $\P^1$, his residue description of the singular locus,
and his tangent space calculation for pairs of zero divisors.  The new
point is that the minor arcs are counted according to the
codimension of their singular loci rather than bounded one at a time.
The complete sum estimate used here is proved in \cite[Theorem 7]{HaseLiuQuadraticBirch}. In the arbitrary hypersurface setting of that paper, multiplication rank and a weighted degeneration of the Jacobian equations give pointwise codimension bounds. Here, the
diagonal form permits a different
refinement: Sawin's description of the singular loci lets us count the
functionals collectively according to codimension, while intersection
theory controls, uniformly in $e$, the degrees of the parameter spaces
occurring in this count.

Here is the numerical reason for doing so. Let $S(\alpha)$ be the complete exponential sum indexed by $\alpha\in P_{de}^\vee$. If $h$ is the codimension of the relevant singular locus, then 
\[|S(\alpha)|\ll_{q,d}(d-1)^hq^{e+1-h/2}.\]We prove that every minor arc has $h\ge e/(d-1)$.  We isolate a class containing every minor arc with $h\le e/2$ and count it directly: the number of functionals in this class with codimension at most $h$ is $O_{q,d}(q^{dh})$. Every remaining minor arc has $h>e/2$, and the number of these with codimension at most $h$ is \begin{equation}\label{eq:intro-count}
 O_d\left(A_d^e q^{dh+1}\right) \text{ with }A_d=\frac{12}{5}d\cdot100^d.
\end{equation}
Thus the ratio between consecutive codimension layers is $(d-1)^nq^{d-n/2}$.  Since the second class begins above $e/2$, the normalized contribution is bounded by \[A_d^e\sum_{h>e/2}((d-1)^nq^{d-n/2})^h\ll A_d^e((d-1)^nq^{d-n/2})^{e/2}=\left(A_d\sqrt{(d-1)^nq^{d-n/2})}\right)^e,\] so it is enough to have
 \[q>(d-1)^{\frac{2n}{n-2d}}A_d^{\frac4{n-2d}}.\]
In particular, if $n\ge(2+\ep)d$ with $\ep>0$, it is enough to take
\[
 q>(d-1)^{2+4/\ep}
 \left(\frac{12}{5}d\cdot100^d\right)^{4/(\ep d)}.
\]
For fixed $\ep>0$, it is therefore enough that $q$ grow polynomially
in $d$: more precisely,
\[
 q>C_\ep d^{2+4/\ep},
\]
where $C_\ep$ depends only on $\ep$.

The degree estimate in \eqref{eq:intro-count} is essential because $q$
is fixed while $e$ tends to infinity. A dimension bound alone does not
give a sufficiently uniform point count: the number and degrees of the
relevant components may grow with $e$. We therefore place the relevant
functionals in closed subsets of $\P(P_{de}^{\vee})$ for which the sum
of the degrees of the irreducible components is $O_d(A_d^e)$. This is
the source of the factor $A_d^e$ in
\eqref{eq:intro-count}; the required degree bound is proved using
intersection theory.

Section \ref{sec:Sawinreduction} recalls Sawin's circle method decomposition and reduces the
theorem to aggregate counting estimates. Sections \ref{sec:supportsandauxiliary} and \ref{sec:minor_arcs} derive the
needed codimension and dimension bounds from the auxiliary congruence
and its tangent spaces. Section \ref{sec:effectiveconstants} makes the count effective by bounding
the degrees of the relevant parameter spaces through intersection
theory, and Section \ref{sec:puttingeverythingtogether} completes the proof.

\begin{ack}
    I'd like to thank Will Sawin and Shuntaro Yamagishi for their interest and helpful comments. Notably, Will also independently arrived at a
    similar strategy: grouping minor arcs by the codimension of their
    associated singular loci and separating common zeros of $a$ and $c$
    whose multiplicities are in the ratio $1\colon d-1$.
    
    ChatGPT 5.6 Sol and 6 Astra were used to proofread this article.
    They also suggested and worked out the computations in Section
    \ref{sec:effectiveconstants} and the proof of Proposition
    \ref{prop:sharpness}. The aggregate minor arc approach was developed
    by the author before this AI-assisted revision. All AI-assisted
    suggestions were independently checked by the author.
\end{ack}
\section{Sawin's circle method reduction}\label{sec:Sawinreduction}
Throughout this section, we assume the hypotheses of Theorem \ref{thm:main}. In particular, $p>d$ and $q>(d-1)^2$.

Fix a non-trivial additive character $\psi\colon\F_q\to\C^\times$.  For
$\alpha\in P_{de}^{\vee}$, put
\[S(\alpha)=\sum_{a\in P_e}\psi(\alpha(a^d)).
\]
By orthogonality, we have
\[N_{e,n}(f)=q^{-(de+1)}
 \sum_{\alpha\in P_{de}^{\vee}}
 \psi(-\alpha(f))S(\alpha)^n.\]
Following \cite[Section 3]{SawinWaring} and \cite[Section 4]{haseliu2024higher}, say that a finite closed
subscheme $Z\subset\P^1$ supports $\alpha$ if $\alpha$ factors through
restriction to $Z$, and let $\deg(\alpha)$ be the minimum possible
length of such a scheme.  We define the \textit{major arcs} on $\P^1$ to be the
functionals with $\deg(\alpha)\le e+1$.  The remaining functionals
are the \textit{minor arcs}.

For a non-zero functional $\alpha$, let
\[
 \Sing_\alpha=\left\{a\in P_e\colon \alpha(a^{d-1}b)=0\text{ for every }b\in P_e\right\}.\] This is the affine singular locus of
$a\mapsto\alpha(a^d)$. Also, let \[h(\alpha)=\codim_{P_e}\Sing_\alpha.\]
Since $p>d,$ polarization and the surjection $\Sym^d P_e\to P_{de}$ show that the pure powers $a^d$ span $P_{de}$. Thus $a\mapsto \alpha(a^d)$ is a non-zero homogeneous polynomial of degree $d$. Applying \cite[Theorem 7]{HaseLiuQuadraticBirch} with $N=e+1$ and $P(a)=\alpha(a^d)$ gives
\begin{equation}\label{eq:aggregate-needed}
 |S(\alpha)|\le \frac{(d-1)^{h(\alpha)}}{1-(d-1)/\sqrt q}q^{e+1-h(\alpha)/2}.
\end{equation}
Indeed, differentiating $P$ in the direction $b$ shows scheme-theoretically that its singular locus is $\Sing_\alpha$. Note also that $h(\alpha)\ge 1$. 
The major arcs and the excess singular series tail contribution require no new
argument. The specialization of
\cite[Lemmas 3.8--3.10]{SawinWaring} to $\P^1$ gives, for every
$0<\delta<(n-d-1)/d$,
\[\begin{split}
 N_{e,n}(f)={}&q^{e(n-d)+n-1}\ell_{\infty,e}(f)\prod_\pi\ell_\pi(f)\\
 &+q^{-(de+1)}
 \sum_{\substack{\alpha\in P_{de}^{\vee}\\\deg(\alpha)>e+1}}\psi(-\alpha(f))S(\alpha)^n\\
 &+O_{n,d,\delta}\left(q^{e(n-d)+n-1}q^{-\delta(e+2)}\right).
\end{split}\]
The only term that has to be analyzed is the middle (minor arc) term.

The following reduction isolates the required geometric input.

\begin{prop}\label{prop:analytic-reduction}
Suppose that the minor arc functionals are divided into two collections with the following properties. In the first collection, every functional has $h(\alpha)\ge e/(d-1)$ and \[\#\{\alpha\colon h(\alpha)\le h\}\ll_{q,d}q^{dh}.\] In the second collection, every functional has $h(\alpha)>e/2$, and, for some $A_d>1$, \[\#\{\alpha\colon h(\alpha)\le h\}\ll_d A_d^e q^{dh+1}.\] Then, if \[q>(d-1)^{\frac{2n}{n-2d}}A_d^{\frac4{n-2d}},\] then there is $\theta=\theta(q,d,n)>0$ such that
\begin{equation}\label{eq:minor-arc-bound}
 q^{-(de+1)}
 \left|\sum_{\substack{\alpha\in P_{de}^{\vee}\\
                   \deg(\alpha)>e+1}}
 \psi(-\alpha(f))S(\alpha)^n\right|
 \ll_{q,d,n,\theta}
 q^{e(n-d)+n-1}q^{-\theta e}.
\end{equation}
\end{prop}

\begin{proof}
Group the non-zero linear functionals by the value of 
$h(\alpha)$ and apply \eqref{eq:aggregate-needed}. Then, the relative
contribution of the first collection is bounded, up to a factor independent of $e$, by
\[\sum_{h\ge e/(d-1)}((d-1)^nq^{d-n/2})^h.\] The relative contribution of the second collection is similarly \[A_d^e\sum_{h>e/2}((d-1)^nq^{d-n/2})^h.\] By the assumption $q$, we have $A_d((d-1)^nq^{d-n/2})^{1/2}<1$, which implies that the common ratio in the two displayed sums is less than one, and, moreover, that after summing the geometric series, both are $O(q^{-\theta e})$ with $\theta > 0$.
This proves the proposition.
\end{proof}

\section{Supports and the auxiliary section}\label{sec:supportsandauxiliary}

Let $0\neq\alpha\in P_{de}^{\vee}$, choose a finite closed subscheme
$Z\subset\P^1$ of minimum length through which $\alpha$ factors,
and write
\[
 m=\deg Z=\deg(\alpha).
\]
By \cite[Lemma 12]{haseliu2024higher}, we have
\[m\le\left\lfloor\frac{de}{2}\right\rfloor+1.\]
The induced functional on $Z$ is primitive, i.e. it doesn't
factor through a proper closed subscheme of $Z$.

Residue duality converts the equations defining $\Sing_\alpha$ into
a congruence relation.  In the notation of
\cite[Lemmas 4.2--4.3]{SawinWaring}, the primitive functional is
represented by an invertible section
\[\widetilde\alpha\in H^0(Z,\omega_{\P^1}(Z)\otimes\O(-de\infty)).\]
For $a\in P_e$, one has $a\in\Sing_\alpha$ iff there is a
unique ``auxiliary''
\[c\in H^0(\P^1,\omega_{\P^1}(Z)\otimes\O(-e\infty))\]
satisfying
\begin{equation}\label{eq:auxiliary-congruence}
 c\rvert_Z=\widetilde\alpha a^{d-1}\rvert_Z.
\end{equation}
After identifying $\omega_{\P^1}(Z)\otimes\O(-e\infty)$ with
$\O((m-e-2)\infty)$, we may write $c\in P_{m-e-2}$.  

The restriction map
\[H^0(\P^1,\omega_{\P^1}(Z)\otimes\O(-e\infty))\to H^0(Z,(\omega_{\P^1}(Z)\otimes\O(-e\infty))\rvert_Z)\]
is injective, since its kernel is $H^0(\P^1,\omega_{\P^1}(-e\infty))=0.$

Residue duality identifies $\Sing_\alpha$ with
the locus where $\widetilde\alpha a^{d-1}\rvert_Z$ lies in the image of
this restriction map.  Choose a linear left inverse of the restriction map.  For
$a\in\Sing_\alpha$, residue duality says that
$\widetilde\alpha a^{d-1}\rvert_Z$ lies in its image, so applying this left
inverse gives the unique section $c$ satisfying
\eqref{eq:auxiliary-congruence}.  This section depends regularly on
$a$, since $\widetilde\alpha a^{d-1}\rvert_Z$ depends polynomially on $a$
and the chosen left inverse is linear.  Consequently, projection onto
$a$ gives an isomorphism
\[
 \{(a,c)\colon c\rvert_Z=\widetilde\alpha a^{d-1}\rvert_Z\}
 \to \Sing_\alpha.
\]
So we can use pairs $(a,c)$ and the irreducible
components and their tangent space dimensions can be identified.

The case $c=0$ can be described explicitly.  If
\[
 Z=\sum_xm_x[x],
\]
let, as in \cite[equation (22)]{SawinWaring},
\[
 D_d(Z)=\sum_x\left\lceil\frac{m_x}{d-1}\right\rceil[x]\text{ and } d_d(Z)=\deg D_d(Z).
\]
Since $\widetilde\alpha$ is invertible on $Z$, the solutions with
$c=0$ (endowed with the reduced induced subscheme structure) form the linear space
\begin{equation}\label{eq:c-zero-space}
 P_e(-D_d(Z)).
\end{equation}
Its codimension is $d_d(Z)$ unless it's just the origin, in which
case its codimension is $e+1$.  This case will be treated separately
below.

\section{Counting the minor arcs}\label{sec:minor_arcs}

The goal of the following two sections is to prove the following proposition.

\begin{prop}\label{prop:aggregate-count}
Assume $p>(d-1)^2$, and let \[A_d=\frac{12}{5}d\cdot100^d.\] Then, every minor arc functional satisfies \[h(\alpha)\ge\frac{e}{d-1}.\] Moreover, the number of $\alpha$ with $h(\alpha)\le h$ such that $c$ vanishes identically on a largest component for a minimum support $Z$ with $2\deg Z < de+2$ is \[O_{q}(q^{dh}).\]

All other minor arc functionals satisfy \[h(\alpha)>\frac{e}{2},\] and the number of these remaining functionals with $h(\alpha)\le h$ is \[O_d(A_d^eq^{dh+1}.)\]
\end{prop}

For the rest of this section, we assume $p>(d-1)^2.$ 

We outline the proof.  Fix $h\le e+1$ and consider the minor arcs $\alpha$ with $h(\alpha)\le h$. Fix one such $\alpha$, a corresponding support
$Z$ of minimum length, and an irreducible component $C$ of
$\Sing_\alpha$ of maximum dimension.  We first keep $Z$ and $\alpha$
fixed and bound $\dim C$.  We then let $Z$ and $\alpha$ vary while
retaining only the numerical data arising from the generic point of
$C$.

If $h=e+1$, the claim follows by taking the whole projective
space $\P(P_{de}^{\vee})$, since its dimension is $de$.
We may therefore assume below that $h\le e$.  In particular, a
largest component of codimension at most $h$ cannot be only the
origin.
\subsection{Removing a common factor from $a$ and $c$}
Let $a$ denote the generic point of $C$, and let $c$ be the auxiliary
section associated with $a$ by \eqref{eq:auxiliary-congruence}.
Suppose first that $a,c\neq0$.  All zero
divisors below are taken on $\P^1$, including possible zeroes at
$\infty$.  Define
\[U=\sum_{\substack{x\notin|Z|\\\ord_x(c)=(d-1)\ord_x(a)>0}}\ord_x(a)[x].\]
Let $R=\deg U$, choose a section $u$ with zero divisor $U$, and write
\[
 a=ug\text{ and } c=u^{d-1}v.\]
The case $U=0$ is allowed, in which case $u=1$ and $R=0$. Let
\[ G=e-R \text{ and }
 H=m-e-2-(d-1)R.\]
 Since $g$ and $v$ are non-zero, $G,H\ge0$; the second formula also
gives $H<m$. Since $u$ is invertible on $Z$, cancellation in
\eqref{eq:auxiliary-congruence} gives
\begin{equation}\label{eq:residual-congruence}
 v\rvert_Z=\widetilde\alpha g^{d-1}\rvert_Z.
\end{equation}

For non-negative integers $i,j,$ and $\nu$, not all zero, let
$\mu(i,j,\nu)$ be the number of geometric points $x$ such that
\[\ord_x(g)=i, \ord_x(v)=j, \text{ and }\mult_x(Z)=\nu.\]
This is essentially the same as Sawin's ``joint-multiplicity'' function for $(g,v)$, except we also track an additional entry $\nu$ recording the multiplicity in $Z$. It
satisfies
\[ \sum_{i,j,\nu} i\mu(i,j,\nu)=G,
 \sum_{i,j,\nu} j\mu(i,j,\nu)=H, \text{ and } \sum_{i,j,\nu} \nu\mu(i,j,\nu)=m.\]Note that no triple with $\nu=0$, $i>0$, and $j=(d-1)i$ occurs, because every
such point is already included in $U$.

Next, define
 \[L(\mu)=
 \sum_{\substack{i,j,\nu\\i+j>0}}
 (i+j-1)\mu(i,j,\nu) \text{ and } s=\sum_{\substack{i,j,\nu\\i+j>0}}\mu(i,j,\nu).\]Then, we have\[
 G+H-L(\mu)
 =
 \sum_{\substack{i,j,\nu\\i+j>0}}
 ((i+j)-(i+j-1))\mu(i,j,\nu)
 =s.
\]
So $s$ is exactly the number of distinct points at which $g$ or
$v$ vanishes.

To compute the dimension of the corresponding stratum of divisor
data, choose, for each triple $(i,j,\nu)$ with $i+j>0$ an ordering of the
$\mu(i,j,\nu)$ points having that triple.  Enumerate all the resulting
points as $x_1,\ldots,x_s$, and let $(i_r,j_r,\nu_r)$ be the triple
attached to $x_r$.  Choosing their positions amounts to choosing a
point of
\[
 (\P^1)^s\setminus
 \bigcup_{r\neq t}\{x_r=x_t\},
\]
which is an open subset of $(\P^1)^s$ and hence has dimension $s$.
Once these positions are chosen, the prescribed multiplicities
determine
\[
 \divisor(g)=\sum_{r=1}^s i_r[x_r] \text{ and } \divisor(v)=\sum_{r=1}^s j_r[x_r],
\]
as well as the part $\sum_{r=1}^s\nu_r[x_r]$ of $Z$ supported at these points.  

Conversely, these divisors recover
the points up to permutations among those having the same triple. Forgetting the chosen orderings has finite fibers and therefore does
not change the dimension. 

The remaining support points of $Z$ are those of type $(0,0,\nu)$
with $\nu>0$. When $Z$ is allowed to vary below, their positions
contribute
\[
 \sum_{\nu>0}\mu(0,0,\nu)
\]
additional parameters.

Let
 \[b(\mu)=
 \sum_{\substack{i+j>0,\ \nu=0,\ j\neq(d-1)i\\
 p\mid((d-1)i-j)/\gcd(i,j)}}\mu(i,j,0).\]
Note that $b(\mu)$ is the number of points away from $Z$ for which
$i+j>0$, $j\neq(d-1)i$, and
\begin{equation}\label{eq:normalized-divisibility}
 p \text{ divides }
 \frac{(d-1)i-j}{\gcd(i,j)}.
\end{equation}
This is exactly the part of Sawin's count in \cite[Lemma 4.7]{SawinWaring} that remains after the roots in $U$ have been separated and the roots on the fixed scheme $Z$ have been omitted.
\subsection{The dimension of the fixed component}

With the notation above, we claim that
\begin{equation}\label{eq:component-dimension} \dim C\le1+\#\Supp(U)+b(\mu)\le R+1+b(\mu).\end{equation} This is essentially the same as the proof of \cite[Lemmas 4.7]{SawinWaring}, but we spell out the details (mostly for our own sake). 

Choose a smooth geometric point $(a,c)$ of a dense open subset of $C$
on which the joint multiplicities of the zero divisors of $a$ and $c$,
as well as their orders at every point of $\Supp(Z)$, are constant. We
use the ``root coordinates'' of \cite[Lemma 4.6]{SawinWaring}. In particular, if $(\dot a,\dot c)$ is a tangent vector at $(a,c)$, differentiating \eqref{eq:auxiliary-congruence} gives
\[\dot c\rvert_Z =(d-1)\widetilde\alpha a^{d-2}\dot a\rvert_Z.\]
It follows that the section
\[a\dot c-(d-1)c\dot a \in H^0(\P^1,\omega_{\P^1}(Z))\]
vanishes on $Z$. It therefore belongs to
$H^0(\P^1,\omega_{\P^1})=0$, and hence
\begin{equation}\label{eq:logarithmic-tangent-identity}\frac{\dot c}{c}=(d-1)\frac{\dot a}{a}.\end{equation}

Let $x$ be a root at which $a$ and $c$ have orders $r$ and $s$, and
let $p^w$ be the greatest power of $p$ dividing $\gcd(r,s)$. Choose a
local parameter $\pi_x$ at $x$. In the
coordinates of \cite[Lemma 4.6]{SawinWaring}, equation (26) there
expresses the two logarithmic derivatives as
\[s_x\left(\frac{r}{p^w}\frac1{\pi_x^{p^w}},
\frac{s}{p^w}\frac1{\pi_x^{p^w}}\right)\]
for a scalar $s_x$. Combining this with
\eqref{eq:logarithmic-tangent-identity} shows that $s_x=0$ unless
$p$ divides $\frac{(d-1)r-s}{\gcd(r,s)}.$

At a root lying in $\Supp(Z)$, the scalar $s_x$ is zero because that root is fixed on the chosen stratum and the orders of vanishing of $a$ and $c$ are constant. Away from $Z$, the roots for which the displayed divisibility can hold consist precisely of the points of $\Supp(U)$ and the roots counted by $b(\mu)$. 

Then, let $T_0$ be the subspace of the tangent space on which $s_x=0$ at the
points of $\Supp(U)$ and at the roots counted by $b(\mu)$. These
conditions give
\[\codim(T_0)\le\#\Supp(U)+b(\mu).\]
For a vector in $T_0$, all the remaining scalars $s_x$ vanish by the
preceding paragraph and \eqref{eq:logarithmic-tangent-identity}. Thus
$\dot a/a$ and $\dot c/c$ have no poles on $\P^1$, so both are constant.
Equation \eqref{eq:logarithmic-tangent-identity} then shows that the
vector is determined by the single constant $\dot a/a$. Hence
$\dim T_0\le1$, and therefore
\[
 \dim C\le1+\#\Supp(U)+b(\mu)
 \le R+1+b(\mu),
\]
as desired.

Next, since $C$ is a component of maximum dimension,
$h(\alpha)=e+1-\dim C$.  Hence
\begin{equation}\label{eq:height-from-component}
 h(\alpha)\ge e-R-b(\mu).
\end{equation}
In particular, if $h(\alpha)\le h$, then
\begin{equation}\label{eq:component-to-height}
 R+b(\mu)\ge e-h.
\end{equation}

We also have the following elementary inequality.  Every pair $(i,j)$ counted
by $b(\mu)$ satisfies
\[
 i+j-1\ge d.
\]
To see this, suppose otherwise and divide $(i,j)$ by its greatest
common divisor, obtaining a coprime pair $(I,J)$.  Since
$j\neq(d-1)i$ and $I+J\le d$, one has
\[
 0<|(d-1)I-J|<(d-1)^2.
\]
If $I=0$, then $J=1$, and if $J=0$, then $I=1$, so the same strict
inequality holds in the two boundary cases.
This contradicts \eqref{eq:normalized-divisibility} because the
characteristic is greater than $(d-1)^2$.  Using this observation and summing over the points
counted by $b(\mu)$ therefore gives
\begin{equation}\label{eq:multiplicity-cost}
 L(\mu)\ge d\,b(\mu).
\end{equation}
\subsection{Letting the functional vary}
The constructions in Section \ref{sec:effectiveconstants} realize the data considered below as locally closed strata in certain parameter spaces. More precisely, for fixed $m,R,\mu,$ the tuples $(Z,[g],[v],[\widetilde \alpha])$ considered below form a locally closed incidence locus in the projective bundle constructed in Section \ref{sec:effectiveconstants}. The  base records the joint divisor data, the fibers of the projective bundle record the line $[\widetilde \alpha]$, and the incidence condition is that $\widetilde \alpha g^{d-1}|_Z$ is proportional to $v\rvert_Z$. Its image in $\P(P_{de}^\vee)$ contains all the classes $[\alpha$ under consideration. We now bound the dimension of this incidence locus.

Multiplying $g$, $v$, or $\widetilde\alpha$ by a non-zero scalar does
not change its zero divisor. From this point onward we
therefore consider these three sections up to scalar, and we also modify \eqref{eq:residual-congruence} to be true up to scalars.  Now fix $m,R$, and $\mu$, but allow $Z$, $[g]$, $[v]$ (we use $[-]$ to denote the section up to scalar), and the primitive functional on $Z$ to vary. 

Now, let $\lambda$ be the non-zero functional on
$H^0(Z,\O_Z(de\infty))$ represented by $\widetilde\alpha$ under
residue duality.  If we define $\alpha(f)=\lambda(f\rvert_Z)$ for all $f\in P_{de},$ then since the restriction map $P_{de}\to H^0(Z,\O_Z(de\infty))$ is surjective by $H^1(\P^1,\O(de\infty-Z))=0$, it follows that $\alpha\ne0$, and
scaling $\widetilde\alpha$ merely scales $\alpha$. We therefore obtain
a well-defined point $[\alpha]\in\P(P_{de}^{\vee})$.

Note that the divisor $U$ is deliberately not included among the varying data.
Its support has already contributed at most $R$ parameters to
\eqref{eq:component-dimension}, while $u^{d-1}$ has cancelled from
\eqref{eq:residual-congruence}.  Nevertheless, removing $U$ lowered
the sum of the degrees of $a$ and $c$ by $dR$.  This is the central
saving.

We next count the remaining parameters.  The distinct points in the
zero divisors of $g$ and $v$ contribute $G+H-L(\mu)$ parameters.  The
support points of $Z$ at which neither section vanishes contribute
\[
 \sum_{\nu>0}\mu(0,0,\nu)
\]
further parameters.  Note that we are considering $g$ and $v$ up to scalar, and a non-zero section of $\O(r)$ on $\P^1$ is determined up to scalar by its degree $r$ zero divisor.

For fixed $Z,g,v$, multiplication by $g^{d-1}$ defines the natural map $H^0\left(Z,\O_Z((m-de-2)\infty)\right)
 \to H^0\left(Z,\O_Z(H\infty)\right).$
At a point where $Z$ has length $\nu$ and $g$ has order $i$, the local
kernel has dimension $\min((d-1)i,\nu)$, since, up to unit, multiplication by $g^{d-1}$ is multiplication by $t^{(d-1)i}$ modulo $t^\nu$.  Moreover, $H<m$, so a
non-zero section of degree $H$ cannot vanish on the length $m$ subscheme
$Z$; hence $v\rvert_Z$ is non-zero.  The inverse image of the line spanned
by $v\rvert_Z$ is a vector space of dimension at most one plus the kernel
dimension.  It follows that the possible points
$[\widetilde\alpha]$ for which
$\widetilde\alpha g^{d-1}\rvert_Z$ is proportional to $v\rvert_Z$ (i.e. \eqref{eq:residual-congruence}) contribute at
most
\[
 \sum_{\substack{i,j,\nu\\i+j>0}}
 \min((d-1)i,\nu)\mu(i,j,\nu)
\]
parameters.  The requirements that the functional be primitive and
that the proportionality is non-zero are open conditions and do
not increase dimension.

The two sums range over disjoint types: the first has $i=j=0$, while
the second has $i+j>0$.  Bound each term
$\mu(0,0,\nu)$ in the first sum by $\nu\mu(0,0,\nu)$, and each term
in the second by $\nu\mu(i,j,\nu)$.  Their total is then at most
$\sum_{i,j,\nu}\nu\mu(i,j,\nu)=m$.  

Combining with the computations from earlier, we therefore have that the family of
the data $(Z,[g],[v],[\widetilde\alpha])$ has dimension at most
\begin{equation}
 G+H-L(\mu)+m=2m-2-dR-L(\mu)\le de-dR-L(\mu),
 \label{eq:varying-dimension}
\end{equation}
where the last inequality follows from
$2m\le de+2$.  If $h(\alpha)\le h$, then
\eqref{eq:component-to-height}, \eqref{eq:multiplicity-cost}, and
\eqref{eq:varying-dimension} show that the corresponding parameter
space has dimension at most
\begin{equation}\label{eq:dh-bound}
 de-dR-L(\mu)
 \le de-d(R+b(\mu))
 \le dh.
\end{equation}
It's also clear that every $\alpha$ under consideration occurs in
such a parameter space.  Fix $\alpha$, a support $Z$ of
minimum length, and a largest irreducible component
$C\subset\Sing_\alpha$ on which the generic auxiliary section $c$ is
non-zero. Then choose a geometric point $a$ in a dense smooth locally closed subset of $C$ on which $a$ and $c$ are non-zero and the joint multiplicity data of their zeroes is constant. Now form $U$ and write $a=ug$ and $c=u^{d-1}v$ as earlier. This determines $m,R,$ and $\mu$, and gives a point $(Z,[g],[v],[\widetilde\alpha])$ of the associated parameter space. Moreover, $\widetilde\alpha$ recovers $\alpha$ via the map from this parameter space to $\P(P_{de}^\vee)$ discussed above. The parameter spaces may
also contain additional points, but this is harmless for an upper bound.
\subsection{The case where the auxiliary section is zero}Suppose that the generic auxiliary section $c$ on a largest component is
zero.  By the graph description following
\eqref{eq:auxiliary-congruence}, the coefficients of $c$ are regular
functions of $a$.  Since they vanish at the generic point of the
component, they vanish identically on it.  Then, by \eqref{eq:c-zero-space} the component is contained in $P_e(-D_d(Z))$.  Conversely, that entire
linear space lies in $\Sing_\alpha$, so by maximality of the irreducible
component we have equality.  Its codimension is $d_d(Z)$ unless it is
just the origin.  But, the latter possibility is not possible because we are working with the case $h(\alpha)\le e$.

We next count these functionals directly, except at the endpoint $2m=de+2.$

First, observe that if $2m<de+2$, then actually the minimum support is unique by \cite[Lemma 13]{haseliu2024higher}. 

So it remains to count a support together with a primitive functional on it. Let $P$ be a closed point. The number of primitive functionals on the divisor $\nu [P]$ is $q^{\nu \deg P}-q^{(\nu-1)\deg P}$. Since primitivity can be checked at the different closed points, by letting $T$ record $d_d(Z)$, we have 
\begin{align*}
    \sum_Z \#\{\ell\colon \ell \text{ primitive on }Z\}T^{d_d(Z)} &= \prod_{P\text{ closed point}}\left(1+\sum_{\nu\ge 1}(q^{\deg P}-1)q^{(\nu-1)\deg P}T^{\lceil \nu/(d-1)\rceil \deg P}\right) \\ 
    &= \prod_{P \text{ closed point}}\frac{1-T^{\deg P}}{1-q^{(d-1)\deg P}T^{\deg P}} \\
    &= \frac{\zeta_{\P^1}(q^{d-1}T)}{\zeta_{\P^1}(T)} \\
    &= \frac{(1-T)(1-qT)}{(1-q^{d-1}T)(1-q^dT)}.
\end{align*}

Observe that the coefficient of $T^t$ in the reciprocal of the denominator is \[q^{dt}\sum_{a=0}^tq^{-a}\ll_q q^{dt}.\]
Multiplying this by the numerator $(1-T)(1-qT)$, the coefficient of $T^t$ still remains $O_{q}(q^{dt}).$ Now, summing all the coefficients ranging from $t=0$ to $t=h$ therefore gives \[O_q(q^{dh}),\] as required by Proposition \ref{prop:aggregate-count}.

Let $\lambda(\nu)$ be the number of points at which $Z$ has
multiplicity $\nu$.  The possible supports $Z$ with these
multiplicities form an open subset of
\[
 \prod_{\nu\ge1}\Sym^{\lambda(\nu)}(\P^1),
\]
of dimension $\sum_\nu\lambda(\nu)$.  For a fixed $Z$, the primitive
functional lines form an open subset of a projective space of
dimension $m-1$.  Hence the total dimension is
\[
 \sum_\nu\lambda(\nu)+m-1.
\]
Note that 
\[
 m=\sum_{\nu\ge1}\nu\lambda(\nu) \text{ and }
 d_d(Z)=\sum_{\nu\ge1}
 \left\lceil\frac{\nu}{d-1}\right\rceil\lambda(\nu),
\]
from which it's clear that 
\[
 m\le(d-1)d_d(Z) \text{ and }
 \sum_\nu\lambda(\nu)\le d_d(Z).
\] So this total dimension is at most $d\,d_d(Z)-1$.  If
$h(\alpha)\le h$, it follows that the dimension of the parameter space is at most \begin{equation} \label{eq:dh-1bound}
    dh-1.
\end{equation}

Conversely, if $\alpha$ arises in this case, choose a minimum support
$Z$ and record its multiplicity function $\lambda$ and its primitive
functional line on $Z$.  Precomposition with restriction again sends
that line to $[\alpha]$.  Thus these families also contain every
$\alpha$ arising this way.
\subsection{The lower bound for a minor arc}Next, we show that \[h(\alpha)\ge \frac{e}{d-1}.\] Again, if $h(\alpha)=e+1$, then the desired inequality is immediate, so we can assume that $h(\alpha)\le e$.

Let $\alpha$ be a
minor arc, so that its minimum support has length
$m>e+1$, and choose an irreducible component
$C\subset\Sing_\alpha$ of maximum dimension.  We use the descriptions
above according to whether the generic auxiliary section $c$ on $C$ is zero or
non-zero. In the case
$c=0$, we have $P_e(-D_d(Z))\neq\{0\}$, so
\[
 h(\alpha)=d_d(Z)\ge\frac{m}{d-1}>
 \frac{e}{d-1}.
\]
In the case $a,c\neq0$, equations
\eqref{eq:height-from-component} and
\eqref{eq:multiplicity-cost} give a stronger estimate.  The divisor
$U$ contributes $dR$ to
$\deg\divisor(a)+\deg\divisor(c)=m-2$. A point counted by $b(\mu)$ has type $(i,j,0)$, where
$i=\ord_x(g)$ and $j=\ord_x(v)$.  Such a point lies outside
$\Supp(U)$, so $u$ is a unit at $x$, which means that $\ord_x(a)=i$ and $\ord_x(c)=j$. It therefore contributes $i+j\ge d+1$
to $\deg\divisor(a)+\deg\divisor(c)$ by the argument proving \eqref{eq:multiplicity-cost}. These points
are distinct, each has one pair of orders $(i,j)$, and all lie outside
$\Supp(U)$, so there is no overlap in the count. Thus
\[
 dR+(d+1)b(\mu)\le m-2\le\frac{de}{2}-1.
\]
It follows that $R+b(\mu)\le e/2-1/d$, and therefore (since $d\ge 3$)
\[
 h(\alpha)\ge e-R-b(\mu)
 \ge\frac e2+\frac1d > \frac e {d-1}.
\]
Moreover, if $c=0$ and $\alpha$ is at the endpoint, i.e. $2m=de+2$, then \[h(\alpha)=d_d(Z)\ge \frac{m}{d-1}>\frac e2,\] as desired.
\section{Effective constants via intersection theory}\label{sec:effectiveconstants}
We make the remaining degree estimate precise. For a projective scheme $V$ over $k$ and a very ample line bundle $\mathcal{A}$, define the cumulative degree \[\cdeg_{\mathcal{A}}V = \sum_{D}\deg_{\mathcal{A}}D,\]where $D$ ranges over all geometrically irreducible components endowed with the reduced induced subscheme structure. When $\mathcal{A}$ is the usual hyperplane bundle, we will omit the subscript.

The reason for using cumulative degree is the following projective point count bound \cite[Theorem 2.1]{LachaudRolland}:\[\#V(\F_q)\le \cdeg(V)(1+q+\cdots + q^{\dim V}).\]

The goal of this section is to prove the following.

\begin{prop}\label{prop:effective-cover}
Assume $p>(d-1)^2$, and let \[A_d = \frac{12}{5}d\cdot100^d.\] For every $e\ge 1$ and $1\le h\le e+1$, there is a projective closed set 
\[
 V_{e,h}\subset\P(P_{de}^{\vee}),
\]
defined over $\F_q$, which contains the class $[\alpha]$ of every
minor arc $\alpha$ satisfying $h(\alpha)\le h$, except those for which
$c$ vanishes identically on a largest component for a minimum support
$Z$ and $2\deg Z<de+2$. Moreover, $\dim V_{e,h}\le dh$ and $\cdeg V_{e,h}\le A_d^e$.
\end{prop}
We first treat the case in which the generic auxiliary section $c$ is non-zero.  Then, in the notation of Section \ref{sec:minor_arcs}, we have
\[ a=ug \text{ and } c=u^{d-1}v\] with $g$ and $v$ non-zero.  Fix $m,R,$ and $\mu$ arising in this case. We construct a closed image containing the corresponding linear functionals. Its degree is bounded by an exponential in $e$ times an explicit multinomial expression in $\mu$. A generation function argument then shows that the sum of these expressions is exponential in $e$. We separately address the endpoint of the case $c=0$ which was not counted directly in Section \ref{sec:minor_arcs}.
\subsection{A projective family for fixed multiplicity data}
Fix $m,R,\mu$ as defined in Section \ref{sec:minor_arcs}. In particular, we have \[\sum_{i,j,\nu}i\mu(i,j,\nu)=e-R=G,\sum_{i,j,\nu}j\mu(i,j,\nu)=m-e-2-(d-1)R=H, \sum_{i,j,\nu}\nu\mu(i,j,\nu)=m,\]as well as $\mu(i,(d-1)i,0)=0$. 

For each $(i,j,\nu)$ with $\mu(i,j,\nu)>0$, choose an effective divisor $D_{i,j,\nu}$ of degree $\mu(i,j,\nu)$. For the open set where these divisors are reduced and pairwise disjoint, they determine \begin{equation}\label{eq:three_divisors}Z=\sum_{i,j,\nu \colon \mu(i,j,\nu)>0}\nu D_{i,j,\nu}, \divisor(g) = \sum_{i,j,\nu \colon \mu(i,j,\nu)>0}iD_{i,j,\nu}, \text{ and }\divisor(v)=\sum_{i,j,\nu \colon \mu(i,j,\nu)>0}jD_{i,j,\nu}.    
\end{equation} Note that a non-zero section on $\P^1$ is determined up to scalar by its zero divisor, so these last two divisors determine the projective classes $[g]$ and $[v]$. 

Letting the divisors vary gives a projective parameter space \[B=\prod_{i,j,\nu \colon \mu(i,j,\nu)>0}\Sym^{\mu(i,j,\nu)}(\P^1)=\prod_{i,j,\nu \colon \mu(i,j,\nu)>0}\P(P_{\mu(i,j,\nu)}).\] In the following, we write \[\bigboxtimes_{i,j,\nu \colon \mu(i,j,\nu)>0}\O(a_{i,j,\nu})\] to be the box product over line bundles corresponding to that of our parameter space $B$ of tuples of divisors.

Next, let $\mathcal{Z}\subset \P^1\times B$ be the universal divisor whose fiber corresponds to the first divisor in \eqref{eq:three_divisors}, and let $\pi\colon \mathcal{Z}\to B$ be the projection. We claim it is finite and flat of degree $m$. Indeed, it is clearly finite because it is proper and has finite fibers. For flatness, note that the universal divisor associated to $D_{i,j,\nu}$ has corresponding line bundle $\O_{\P^1}(\mu(i,j,\nu)\infty)\boxtimes \O(1)$. The universal divisor sequence associated to $\mathcal{Z}$ is hence \begin{equation}\label{eq:universal_divisor_sequence}
    0\to \O_{\P^1}(-m\infty)\boxtimes \bigboxtimes_{\mu(i,j,\nu)>0}\O(-\nu)\to \O\to \O_\mathcal{Z}\to 0.
\end{equation} Pushing this forward along $\pi$ gives the exact sequence \[0\to \O_{B}\to \pi_*\O_{\mathcal{Z}}\to H^1(\P^1,\O(-m))\otimes \bigboxtimes_{\mu(i,j,\nu)>0}\O(-\nu)\to 0,\]which shows that $\pi_*\O_\mathcal{Z}$ is locally free of rank $m$. This implies that $\pi$ is finite flat of degree $m$.

Twisting \eqref{eq:universal_divisor_sequence} and pushing forward along $\pi$ gives the universal restriction \[P_{de}\otimes \O_B\to \pi_* \O_{\mathcal{Z}}(de\infty),\] which is a surjection because $de-m\ge -1$. Using the convention that projective bundles parametrize lines, consider \begin{equation}\label{eq:projectivebundleparam}
    \P_{B}((\pi_*\O_{\mathcal{Z}}(de\infty))^\vee).
\end{equation} This parametrizes the divisor data together with a line \[[\ell]\in \P(H^0(Z,\O_Z(de\infty))^\vee).\]

Now, dualizing the universal restriction map gives a closed immersion of \eqref{eq:projectivebundleparam} into the product $\prod_{i,j,\nu \colon \mu(i,j,\nu)>0}\P(P_{\mu(i,j,\nu)})\times \P(P_{de}^\vee)$. Projecting this down to $\P(P_{de}^\vee)$ gives a proper map which on points is given by \begin{equation}\label{eq:projectivefamilybeforeincidence}
    (D_{i,j,\nu},[\ell])\mapsto [\ell \circ \res_Z].
\end{equation}

\subsection{Imposing a closed incidence relation}
In \eqref{eq:projectivefamilybeforeincidence}, it remains to impose the condition $\widetilde{\alpha}g^{d-1}\rvert_Z = v\rvert_Z$. Since we only care about the projective classes $[g]$ and $[v]$, this equality translates to an equality up to non-zero scaling. 

For a fixed $Z$, residue duality identifies the $[\ell]$ with $[\widetilde \alpha]$, which we next discuss in families (i.e. as $Z$ varies). Let $\mathcal{L}$ be a line bundle on $\mathcal{Z}$ and $\omega_\pi$ be the relative canonical sheaf.  By Grothendieck duality, we have \[(\pi_* \mathcal{L})^\vee \cong \pi_* (\mathcal{L}^{-1}\otimes \omega_\pi).\] Moreover, by adjunction, we have \[\omega_\pi\cong \O_\mathcal{Z}((m-2)\infty)\otimes \bigboxtimes_{i,j,\nu\colon \mu(i,j,\nu)>0}\O(\nu).\] For $\mathcal{L}=\O_\mathcal{Z}(de\infty)$, we then get  \[(\pi_*\O_{\mathcal{Z}}(de\infty))^\vee\cong \pi_*\O_{\mathcal{Z}}((m-de-2)\infty)\otimes \bigboxtimes_{\mu(i,j,\nu)>0}\O(\nu).\] This is the family version of the residue duality discussed earlier.

Next, let $\O(-1)$ is the tautological line of linear functionals on \eqref{eq:projectivebundleparam}. Let us explain what $\widetilde{\alpha}g^{d-1}\rvert_Z$ and $v\rvert_Z$ correspond to. For $\widetilde{\alpha}g^{d-1}\rvert_Z$, note that the multiplication by $g^{d-1}$ gives the map:
\begin{equation}\label{eq:leftside}
    \O(-1)\otimes \bigboxtimes_{\mu(i,j,\nu)>0}\O(-(d-1)i)\to \pi_*\O_{\mathcal{Z}}((m-e-2-(d-1)R)\infty)\otimes \bigboxtimes_{\mu(i,j,\nu)>0}\O(\nu).
\end{equation} 
More precisely, the line spanned by $g$ corresponds to
\[\bigboxtimes_{\mu(i,j,\nu)>0}\O(-i),
\]
which is equipped with the universal evaluation map
\[\bigboxtimes_{\mu(i,j,\nu)>0}\O(-i)\to \O_{\P^1}((e-R)\infty).
\]
Relative residue duality gives a line of sections
\[\O(-1)\to\pi_*\O_{\mathcal{Z}}((m-de-2)\infty)\otimes \bigboxtimes_{\mu(i,j,\nu)>0}\O(\nu).\]
For $v\rvert_Z$, after tensoring by $\boxtimes_{\mu(i,j,\nu)>0}\O(\nu)$ to match the target of the previous map, we similarly have \begin{equation}\label{eq:rightside} \bigboxtimes_{\mu(i,j,\nu)>0}\O(\nu-j)\to \pi_*\O_{\mathcal{Z}}((m-e-2-(d-1)R)\infty)\otimes \bigboxtimes_{\mu(i,j,\nu)>0}\O(\nu).\end{equation}
Two vectors are proportional when their wedge is zero: wedging the two maps \eqref{eq:leftside} and \eqref{eq:rightside}, we obtain a section of \begin{equation}\label{eq:wedgebundle}
    \O(1) \otimes \bigboxtimes_{\mu(i,j,\nu)>0}\O((d-1)i+j-\nu)\otimes \bigwedge^2\left(\pi_*\O_{\mathcal{Z}}((m-e-2-(d-1)R)\infty)\otimes \bigboxtimes_{\mu(i,j,\nu)>0}\O(\nu)\right).
\end{equation}
In particular, the zero scheme is the family version of the locus encoding this proportionality. There are a few open conditions inside this zero scheme: 1. the divisors $D_{i,j,\nu}$ must be reduced and pairwise disjoint, 2. the functional must be primitive, and 3. the map \eqref{eq:leftside} should be non-zero. 

Take the reduced closure of this open locus, and after projecting via the map described in \eqref{eq:projectivefamilybeforeincidence}, denote the reduced image by \[W_{m,R,\mu}\subset \P(P_{de}^\vee).\]

\subsection{Two general degree bounds} We prove two general facts about degree. The first is a version of Bezout's theorem for the zero scheme of the section of a vector bundle. 
\begin{lem}\label{lem:Bezout}
    Let $T$ be an integral projective scheme of dimension $N$ over an algebraically closed field, $L$ be a very ample line bundle on $T$, and $F$ be a vector bundle on $T$. Suppose $F^\vee \otimes L^{\otimes b}$ is globally generated for some positive $b$. Then, any section $s$ of $F$ satisfies \[\cdeg_LZ(s)\le \sum_{j=0}^Nb^j\deg_L T\le (N+1)b^N\deg_L T.\]
\end{lem}
\begin{proof}
    By global generation we have a surjection $H^0(T,F^\vee \otimes L^{\otimes b})\to F^\vee \otimes L^{\otimes b}$, which is locally split. Dualizing and twisting by $L^\otimes b$ then gives a locally split injection \[F\hookrightarrow H^0(T,F^\vee \otimes L^{\otimes b})^\vee \otimes L^{\otimes b}.\] This means that we can view $Z(s)$ actually as the common zero scheme of some number of sections $s_1,\ldots, s_M$ of $L^{\otimes b}$. 

    To run the Bezout argument, we now cut successively by these sections. We first recall the form of Bezout used here: If $Y\subset T$ is an integral closed subscheme and a section of $L^{\otimes b}\rvert_Y$ is not identically zero, then it intersects $Y$ properly. Its zero divisor has class \[c_1(L^{\otimes b})\cap [Y] = bc_1(L)\cap [Y],\] from which it's clear (by forgetting the multiplicities recorded in the corresponding cycle) that the sum of the $L$-degrees of the reduced components of the zero divisor is at most \[b\deg_L Y.\] We apply this observation repeatedly. Start with $T$. If it is contained in $Z(s)$, then we stop. Else, consider the components that have not already been recorded and are not contained in $Z(s)$. At each generic point of such a component, at least one of the sections $s_i$ will be non-zero. We can then choose a linear combination of these $s_i$ that is non-zero at the generic points of all of these components (since there are finitely many). So each component is intersected with properly. 

    After the first cut, Bezout says the sum of the degrees of the resulting components is at most $b\deg_L T$. Record the components that are contained in $Z(s)$, discard them from the process above, and cut the remaining components by another linear combination chosen in the same way. If their total degree before this second cut is at most $b\deg_L T$, then Bezout bounds the next batch by $b^2\deg_L T$. Continuing this process, at the $i$th stage, the total degree of the components is at most $b^i \deg_L T$. Each cut decreases the dimension by one, so this process will stop after $N$ cuts. Summing the contributions over all steps then gives the upper bound \[\cdeg_L Z(s) \le \sum_{i=0}^Nb^i\deg_L T,\]as desired.
\end{proof}
The second fact bounds the degree of a projective image in terms of a
degree on the source.
\begin{lem}\label{lem:cumdegacrossmap}
    Let $Y$ be an integral projective scheme of dimension $c$ over an algebraically closed field, and let $f\colon Y\to \P^N$ be a morphism. Put \[\eta = c_1(f^*\O_{\P^N}(1)).\] If $\zeta$ is nef and $\eta+\zeta$ is very ample, then the (reduced, closed) image satisfies \[\deg f(Y)\le \int_Y(\eta+\zeta)^c.\] When $Y$ has multiple components, the cumulative degree of the image is bounded by the sum of the right-hand side over the components of the source.
\end{lem}
\begin{proof}
    Let $r=\dim f(Y)$. For $r=0$, the claim is immediate. Otherwise, the restriction of $\eta$ to a fiber of $f$ is trivial and the restriction of $\eta + \zeta$ remains very ample. So, we have \[f_*((\eta +\zeta)^{c-r}\cap [Y])=a[f(Y)]\] for some positive $a$; in fact, $a$ is the degree of the generic fiber with respect to $\eta + \zeta$, and is equal to $(\eta + \zeta)^{c-r}$. 

    Now, the projection formula gives \[(c_1(\O_{\P^N}(1))\rvert_{f(Y)})^r\cap f_*(\gamma)=f_*(f^*(c_1(\O_{\P^N}(1))\rvert_{f(Y)})^r\cap\gamma),\] so for $\gamma=(\eta+\zeta)^{c-r}\cap [Y]$ and taking degrees, we get \[\deg f(Y) \le a \deg f(Y) = \int_Y \eta^r (\eta + \zeta)^{c-r}.\] Next, since $\eta$ and $\zeta$ are nef, we have \[(\eta+\zeta)^r - \eta^r=\zeta((\eta+\zeta)^{r-1}+\cdots + \eta^{r-1})\] has non-negative intersection with $(\eta+\zeta)^{c-r}$. So, we have \[\deg f(Y)\le \int_Y(\eta+\zeta)^c.\] The result follows.
\end{proof}

\subsection{The degree contribution for non-zero auxiliary section} We fix the following very ample polarization on \eqref{eq:projectivebundleparam}: \begin{equation}\label{eq:amplelinebundle}
    \O(1)\otimes\bigboxtimes_{\mu(i,j,\nu)>0}\O((d-1)i+j+\nu).
\end{equation} 
This particular choice will be made clearer in the calculations below, but one reason is because \[\sum_{i,j,\nu}((d-1)i+j+\nu)\mu(i,j,\nu)=(d-2)e+2m-2-2(d-1)R\le 2(d-1)e.\]
Let $H_{i,j,\nu}$ be the hyperplane class corresponding to the $(i,j,\nu)$ factor of  the base \[B=\prod_{\mu(i,j,\nu)>0}\P^{\mu(i,j,\nu)},\] so that degree of the base with respect to $\boxtimes_{\mu(i,j,\nu)>0}\O((d-1)i+j+\nu)$ is \begin{equation}
\begin{aligned}
 &\int_B\left(\sum_{i,j,\nu}((d-1)i+j+\nu)H_{i,j,\nu}\right)^{\sum_{i,j,\nu}\mu(i,j,\nu)}\\
 &\qquad=
 \frac{(\sum_{i,j,\nu}\mu(i,j,\nu))!}
 {\prod_{\mu(i,j,\nu)>0}\mu(i,j,\nu)!}\times
 \prod_{\mu(i,j,\nu)>0}
 ((d-1)i+j+\nu)^{\mu(i,j,\nu)}.
\end{aligned}
\end{equation}
Next, we calculate the degree on the projective bundle.
\begin{lem}\label{lem:degofprojectivebundleparam}
The degree of \eqref{eq:projectivebundleparam} with respect to
\eqref{eq:amplelinebundle} is at most
\[
 4^{de}
 \frac{(\sum_{i,j,\nu}\mu(i,j,\nu))!}
 {\prod_{\mu(i,j,\nu)>0}\mu(i,j,\nu)!}
 \prod_{\mu(i,j,\nu)>0}
 ((d-1)i+j+\nu)^{\mu(i,j,\nu)}.
\]
\end{lem}
\begin{proof}
Let
\[E=(\pi_*\O_{\mathcal Z}(de\infty))^\vee, r=\dim B=\sum_{i,j,\nu}\mu(i,j,\nu), \text{ and } N=\bigboxtimes_{\mu(i,j,\nu)>0}\O(\nu).\]
Also, let $\pi\colon\P_B(E)\to B$ be the obvious projection.
Twisting \eqref{eq:universal_divisor_sequence} by
$\O_{\P^1}(de\infty)$ and pushing it forward along $\pi$ gives
\[0\to P_{de-m}\otimes N^{-1}\to P_{de}\otimes\O_B\to E^\vee\to 0.\]
Consequently, if $x=c_1(N)$, then by basic properties of Chern/Segre classes, we have
\[c(E)=(1+x)^{-(de-m+1)} \text{ and }s_k(E)=\binom{de-m+1}{k}x^k.\]

The projective bundle has dimension $m+r-1$. Let $\xi=c_1(\O_{\P(E)}(1))$ and let
\[h=c_1\left(\bigboxtimes_{\mu(i,j,\nu)>0}\O((d-1)i+j+\nu)\right).\]
By the projective bundle formula, we obtain
\[\int_{\P(E)}(\xi+h)^{m+r-1}=\sum_{k=0}^r\binom{m+r-1}{m+k-1}\binom{de-m+1}{k}\int_Bx^kh^{r-k}.\]
Since $h-x$ is visibly nef, we have
\[\int_Bx^kh^{r-k}\le\int_Bh^r.
\]
It follows that
\begin{align*}
\int_{\P(E)}(\xi+h)^{m+r-1} &\le \left(\int_Bh^r\right)\sum_{k=0}^r\binom{m+r-1}{m+k-1}\binom{de-m+1}{k}\\
 &\le 2^{m+r-1}2^{de-m+1}\int_Bh^r =2^{de+r}\int_Bh^r.
\end{align*}
 
Moreover,
\[r\le\sum_{i,j,\nu}(i+j+\nu)\mu(i,j,\nu)=G+H+m=2m-2-dR\le de.\]
Therefore
\[\int_{\P(E)}(\xi+h)^{m+r-1}\le4^{de}\int_Bh^r.\]
Substituting the preceding formula for $\int_Bh^r$ proves the result.
\end{proof}

Next, we check that the wedge section from earlier satisfies the conditions of Lemma \ref{lem:Bezout}.

\begin{lem}\label{lem:checkconditions}
Let
\[
 L=\O(1)\otimes
 \bigboxtimes_{\mu(i,j,\nu)>0}\O((d-1)i+j+\nu).
\]
Then $F^\vee\otimes L$ is globally generated, where $F$ is the wedge
bundle from \eqref{eq:wedgebundle}.
\end{lem}
\begin{proof}
Let
\[N=\bigboxtimes_{\mu(i,j,\nu)>0}\O(\nu) \text{ and }H=m-e-2-(d-1)R.\]
Here $H\ge0$ and $H-m\le-2$. Twisting
\eqref{eq:universal_divisor_sequence} by $\O_{\P^1}(H\infty)$ and
pushing it forward along $\pi$ gives
\[0\to P_H\otimes\O_B
 \to\pi_*\O_{\mathcal Z}(H\infty)
 \to H^1(\P^1,\O(H-m))\otimes N^{-1}
 \to 0.\]
After tensoring by $N$, this becomes
\[0\to N^{\oplus(H+1)}
 \to\pi_*\O_{\mathcal Z}(H\infty)\otimes N
 \to\O_B^{\oplus(m-H-1)}
 \to0.
\]
Its extension class lies in copies of $H^1(B,N)$, which vanishes, so
the sequence splits.

It follows that
$\bigwedge^2(\pi_*\O_{\mathcal Z}(H\infty)\otimes N)$ is a direct sum
of copies of $N^{\otimes 2}$, $N$, and $\O_B$. Thus $F$ is a direct sum of copies
of
\[L,\O(1)\otimes\bigboxtimes\O((d-1)i+j),\text{ and }\O(1)\otimes\bigboxtimes\O((d-1)i+j-\nu).
\]
After taking duals and tensoring by $L$, these become respectively
$\O_B$, $N$, and $N^{\otimes 2}$, all of which are globally generated.
\end{proof}
We can now bound the cumulative degrees of $W_{m,R,\mu},$ as defined earlier.
\begin{prop}\label{prop:firstcasecdeg}
    We have \[\cdeg W_{m,R,\mu}\le (2de+1)4^{de}\frac{(\sum_{i,j,\nu}\mu(i,j,\nu))!}{\prod_{\mu(i,j,\nu)>0}\mu(i,j,\nu)!}\prod_{\mu(i,j,\nu)>0}((d-1)i+j+\nu)^{\mu(i,j,\nu)}.\]
\end{prop}
\begin{proof}
    The projective bundle \eqref{eq:projectivebundleparam} is integral and has dimension $m-1+\sum_{i,j,\nu}\mu(i,j,\nu)\le 2de$ (this bounding is a bit coarse). 

    Apply Lemma \ref{lem:Bezout} to the wedge section with the polarization \eqref{eq:amplelinebundle} and $b=1$, using Lemma \ref{lem:checkconditions} to verify the assumptions, to get an upper bound on the cumulative degree of the zero locus of wedge section in terms of the degree of \eqref{eq:projectivebundleparam}. This degree was bounded in Lemma \ref{lem:degofprojectivebundleparam}. Finally, using Lemma \ref{lem:cumdegacrossmap} with $\eta+\zeta$ the restriction of the polarization \eqref{eq:amplelinebundle} to any irreducible component of the wedge locus, we obtain the desired bound on $\cdeg W_{m,R,\mu}$.
\end{proof}

Finally, we sum over all of the valid multiplicity data. 

\begin{lem}\label{lem:sum-multiplicity-degrees} The sum \[\sum_{\substack{\mu\colon \Z_{\ge 0}^3\to \Z_{\ge 0}\\\sum_{i,j,\nu}((d-1)i+j+\nu)\mu(i,j,\nu)\le 2(d-1)e}}\frac{(\sum_{i,j,\nu}\mu(i,j,\nu))!}{\prod_{\mu(i,j,\nu)>0}\mu(i,j,\nu)!}\prod_{\mu(i,j,\nu)>0}((d-1)i+j+\nu)^{\mu(i,j,\nu)}\le 20\cdot 25^{(d-1)e}.\]
\end{lem}
\begin{proof}
    Consider the generating function \[B(x)=\sum_{i,j,\nu \ge 0, (i,j,\nu)\ne 0}((d-1)i+j+\nu)x^{(d-1)i+j+\nu}=x\frac d{dx}\frac{1}{(1-x^{d-1})(1-x)^2}.\] Then, $B(x)^r$ records an ordered list of $r$ triples. By forgetting the order and recording only how many times each triple occurs (namely $\mu(i,j,\nu)$), and then multiplying by the number of orderings again, we have \[\sum_\mu \frac{(\sum\mu(i,j,\nu))!}{\prod \mu(i,j,\nu)!}\prod((d-1)i+j+\nu)^{\mu(i,j,\nu)}x^{\sum((d-1)i+j+\nu)\mu(i,j,\nu)}=\sum_r B(x)^r = \frac 1 {1-B(x)}.\]
    At $x=1/5$, observe that we have \[x\frac d {dx}\frac{1}{(1-x^{d-1})(1-x)^2}=\frac{2x+(d-1)x^{d-1}-(d+1)x^d}{(1-x^{d-1})^2(1-x)^3}=\frac{\frac 25+\frac {4d-6}{5^d}}{\frac {64}{125}(1-5^{1-d})^2}.\] For $d\ge 3$, this is maximized at $d=3$, where the value is slightly under $0.95$.

    Write $1/(1-B(x)) = \sum_{k \ge 0}c_k x^k$ and observe that $c_k \ge 0$. Then, we have \[\sum_{k \ge 0} c_k 5^{-k}=\frac{1}{1-B(1/5)}<\frac{1}{1-0.95}=20.\] In particular, our desired sum is \[c_0+\cdots +c_{2(d-1)e}\le 5^{2(d-1)e}\sum_kc_k5^{-k}<20\cdot 25^{(d-1)e},\] as desired.
\end{proof}
\subsection{The case where the auxiliary section is zero}
When the auxiliary section $c$ is zero, the case where $2m<de+2$ was handled in Section \ref{sec:minor_arcs}, so it remains to address the case $2m=de+2$. 

Fix $\lambda(\nu)$ satisfying $\sum_\nu \nu\lambda(\nu)=m$. Then, over the space \[B_{\lambda}=\prod_{\lambda(\nu)>0}\Sym^{\lambda(\nu)}(\P^1),\] let $\mathcal{Z}=\sum_\nu \nu D_\nu$ be the universal divisor and $\pi$ be the projection. Again, restriction gives a proper map 
\[\P_{B_{\lambda}}((\pi_*(\O_{\mathcal Z}(de\infty))^\vee)\to \P(P_{de}^\vee).\] 
Let $W_{m,\lambda,0}$ be its scheme-theoretic image. For fixed $Z$, every primitive functional on $Z$ has the same $c=0$ locus $P_e(-D_d(Z))$, so we do not need to impose a further incidence condition.

If this locus has codimension at most $h$, then we have \[m\le (d-1)d_d(Z) \text{ and } \sum_{\nu\ge 1}\lambda(\nu)\le d_d(Z)\le h.\] 
The projective bundle $\P_{B_\lambda}((\pi_*(\O_{\mathcal Z}(de\infty))^\vee)$ above hence has dimension \[m+\sum_{\nu\ge 1}\lambda(\nu)-1\le dh-1.\]  

The same calculations of the previous subsection with the polarization $\O(1)\otimes \boxtimes_{\lambda(\nu)>0}\O(\nu)$ give \[\cdeg W_{m,\lambda,0}\le 4^{de} \frac{(\sum_{\nu \ge 1}\lambda(\nu))!}{\prod_{\lambda(\nu)>0}\lambda(\nu)!}\prod_{\lambda(\nu)}\nu^{\lambda(\nu)}.\]
Moreover, the expression on the right is obtained from the summand in Lemma
\ref{lem:sum-multiplicity-degrees} by setting
$\mu(0,0,\nu)=\lambda(\nu)$ and setting all other values of $\mu$ equal to zero---this is just a
combinatorial identification. Since the corresponding total
weight is
\[\sum_{\nu\ge1}\nu\lambda(\nu)=m=\frac{de+2}{2}\le 2(d-1)e,
\]
it follows that the sum of these expressions is also bounded by the same
generating function estimate in Lemma
\ref{lem:sum-multiplicity-degrees}.
\subsection{The total degree}
\begin{prop}\label{prop:totdeg}
    For $A_d = \frac{12}{5}d\cdot 100^d$, we have \[\sum_{m,R,\mu}\cdeg W_{m,R,\mu}+\sum_{m,\lambda}\cdeg W_{m,\lambda,0}\le A_d^e,\] where the first sum ranges over valid data where the auxiliary section $c$ is non-zero and the second sum ranges over the case $2m=de+2$ with $c=0$.
\end{prop}
\begin{proof}
    By Proposition \ref{prop:firstcasecdeg} and Lemma \ref{lem:sum-multiplicity-degrees}, we have the following upper bound for the first sum: \[\sum_{m,R,\mu}\cdeg W_{m,R,\mu}\le 20(2de+1)4^{de}25^{(d-1)e}.\] Similarly, for the second sum, we have \[\sum_{m,\lambda}\cdeg W_{m,\lambda,0}\le20\cdot 4^{de}25^{(d-1)e},\] so the combined sum is bounded above by \[20(2de+2)4^{de}25^{(d-1)e}.\] Since $2de+2\le (3d)^e$ and $20\le 20^e$, so this is at most $A_d^e$ with the choice of \[A_d = \frac{12}{5}d\cdot 100^d,\] as desired.
\end{proof}

Finally, we are able to prove the main proposition of this section.

\begin{proof}[Proof of Proposition \ref{prop:effective-cover}]
    For $h=e+1$, take $V_{e,h}=\P(P_{de}^\vee)$. Then, its dimension is $de\le d(e+1)$ with cumulative degree one, so we may assume $h\le e$ for the rest of this proof.

    For $a,c\ne 0$, consider every $(m,R,\mu)$ satisfying the degree conditions above and $R+b(\mu)\ge e-h$. For the case $c=0$, consider only the case $2m=de+2$, $\sum_\nu \lambda(\nu)=m$, and $\sum_\nu \lceil \frac \nu {d-1}\rceil \lambda(\nu)\le h$. Then, define
    \[V_{e,h} = \bigcup_{m,R,\mu}W_{m,R,\mu}\cup \bigcup_{m,\lambda}W_{m,\lambda,0}\]

    Coverage of all considered $\alpha$ follows directly from the construction.  Indeed, choose a minimum support $Z$ and a largest component of $\Sing_\alpha$.  If the auxiliary section is generically non-zero, the factorization $a=ug$ and  $c=u^{d-1}v$ places $[\alpha]$ in one of the sets $W_{m,R,\mu}$. If it is identically zero and was not included in the direct count, then $2m=de+2$, and the point $[\alpha_Z]\in\P(H^0(Z,\O_Z(de\infty))^\vee)$ defined by $\alpha$ lies in the projective bundle used to define one of the sets $W_{m,\lambda,0}$. Thus the union defining $V_{e,h}$ contains every required class $[\alpha]$.

    Regarding the dimension, in the non-zero case \eqref{eq:dh-bound} gives the upper bound $dh$, and in the zero case, \eqref{eq:dh-1bound} gives the upper bound $dh-1$. So $\dim V_{e,h}\le dh$. 

    Finally, Proposition \ref{prop:totdeg} addresses the claim about cumulative degree, and so the result follows.
\end{proof}
We can now prove the main proposition of the previous section.
\begin{proof}[Proof of Proposition \ref{prop:aggregate-count}] The previous section already established all but the last claim.

By Proposition \ref{prop:effective-cover} and \cite[Theorem 2.1]{LachaudRolland}, we have \[\#V_{e,h}(\F_q)\le A_d^e(1+q+\cdots +q^{dh}).\]
Multiplying by $q-1$ accounts for the discrepancy between a projective class $[\alpha]$ versus its non-zero representative $\alpha$. In other words, there are \[O_{d}(A_d^eq^{dh+1})\] ``remaining functionals'' with $h(\alpha)\le  h$.
    
\end{proof}
\section{Putting everything together}\label{sec:puttingeverythingtogether}
In this section, we prove Theorem \ref{thm:main} and Proposition \ref{prop:sharpness}.

\begin{proof}[Proof of Theorem \ref{thm:main}]
Let $A_d$ be the constant in Proposition \ref{prop:aggregate-count}. The assumed inequality for $q$ is precisely \[A_d\sqrt{(d-1)^nq^{d-n/2}}<1.\]
We now apply Proposition \ref{prop:analytic-reduction}. Specifically, take the first collection to be the minor arc functionals for which $c$ vanishes identically on a largest component for a minimum support $Z$
and $2\deg Z<de+2$---these are counted directly in Proposition \ref{prop:aggregate-count}. Take all remaining minor arcs as the second collection. Then, we get \eqref{eq:minor-arc-bound} for some $\theta > 0$.

Now, let $\delta = (n-d-1)/(2d)>0$ and decrease $\theta$, if necessary, so that $\theta < \delta$. Then, both the minor arc estimate and the excess tail error are both \[O_{q,d,n,\theta}(q^{e(n-d)+n-1}q^{-\theta e}).\] This establishes the error term in \eqref{eq:main-asymptotic}.

It remains to check that the product of local factors is bounded above
and below by positive constants depending only on $q,d,$ and $n$.
This will show that the error above is a power saving relative to the
main term, uniformly in $e$ and $f$.

We use the calculation in the proof of
\cite[Lemma 3.11]{SawinWaring}, retaining its sharper numerical form.
Its hypotheses $p\nmid d$, $n>d+1$, and $n\ge5$ all hold here.  If
$v\in|\P^1|$ is a closed point and $Q=q^{\deg v}$, write
$\ell_v(f)$ for its local factor; at $v=\infty$ this means
$\ell_{\infty,e}(f)$.  That calculation bounds the amount by which
$\ell_v(f)$ can fall below one by
\[\frac{Q-1}{1-Q^{d-n}}\left((d-1)^nQ^{-n/2}+\sum_{r=2}^dQ^{r-1-n}\right).\]
Our choice of $q$ implies $(d-1)^nq^{d-n/2}<1$.  Since $Q\ge q$ and
$d-n/2<0$, we also have $(d-1)^nQ^{d-n/2}<1$.  The first term in the
last display is therefore at most
\[\frac{Q^{1-d}}{1-Q^{d-n}},
\]
and the remaining terms are at most
\[\frac{Q^{d-n}}{(1-Q^{-1})(1-Q^{d-n})}.
\]
Moreover, $Q\ge5$, $Q^{1-d}\le Q^{-2}$, and
$Q^{d-n}\le Q^{-4}$.  The sum of the last two displays is therefore
less than $1/2$.  Thus every local factor is uniformly bounded away
from zero.  The estimate in the same proof,
\[\ell_v(f)=1+O_{d,n}\left(Q^{1-n/2}+Q^{d-n}\right),
\]
is summable over the closed points of $\P^1$.  It follows that
\[\left(\ell_{\infty,e}(f)\prod_{\pi\text{ monic irreducible}}\ell_\pi(f)\right)^{-1}=O_{q,d,n}(1),
\]
uniformly in $e$ and $f$.  Thus the error is also a power saving
relative to the main term.  This proves the theorem.\end{proof}

\begin{proof}[Proof of Proposition \ref{prop:sharpness}] For $n=2d$ or $2d-1$, we set $C_{n,e}(f)$ to denote the product of local factors in $n$ variables in the main term of \eqref{eq:main-asymptotic}. For a finite subscheme $Z\subset \P^1$ and for any $\beta\in H^0(Z,\O_Z(de\infty))^\vee$, let \[S_Z(\beta)=q^{-\deg Z}\sum_{b\in H^0(Z,\O_Z(e\infty))}\psi(\beta(b^d)).\]
Again, we call $\beta$ primitive if it does not factor through any proper subscheme of $Z$. For an integer $B$, let \[C_{n,e}^{\le B}(f)=\sum_{\substack{Z\subsetneq \P^1 \\ \deg Z\le B}}\sum_{\beta\text{ primitive}}S_Z(\beta)^n\psi(-\beta(f\rvert_Z)).\] 
Note that the full series (without the restriction on $B$) is $C_{n,e}(f)$ by \cite[Lemma 3.7]{SawinWaring} and \cite[Lemma 3.8]{SawinWaring}. 

We next claim that the contribution from subschemes $Z$ with $\deg Z > B$ tends to zero as $B\to\infty,$ uniformly in $e$ and $f$. At a closed point $v$, let $Q=q^{\deg v}$, and consider \[1+\sum_{m\ge1} \sum_{\substack{\beta\text{ primitive}\\\text{on }m[v]}} |S_{m[v]}(\beta)|^n.\]
There are at most $Q^m$ primitive functionals on $m[v]$. Writing $m=kd+j$ with $1\le j\le d$, the bounds of \cite[Lemma 3.6]{SawinWaring} and summing over $k$ show that the preceding expression is \[1+O_d(Q^{1-n/2}+Q^{d-n}).\]
By our assumption on $n$, both exponents are less than $-1$. Since there are $O(q^m)$ closed points of degree $m$, it follows that \[\sum_{Z\subsetneq \P^1}\sum_{\beta \text{ primitive}}|S_Z(\beta)|^n<\infty\] by using the multiplicativity of $S_Z$ to express the series as a product of the above local terms. Moreover, the bound is uniform in $e$ because we may choose compatible trivializations $\O(e\infty)|_Z\cong \O(de\infty)|_Z\cong \O_Z$, so that the numbers $S_Z(\beta)$ are independent of $e$. Since $\psi(-\beta(f\rvert_Z))$ has magnitude one, it follows, uniformly in $e$ and $f$ that 
\begin{equation}\label{eq:uniformbound}
    |C_{n,e}(f)-C_{n,e}^{\le B}(f)|\le \sum_{\deg Z>B}\sum_{\beta\text{ primitive}}|S_Z(\beta)^n|\to0.
\end{equation}
When $f=0$, $\psi(-\beta(f\rvert_Z))$ is moreover equal to one, so this sum and its truncated version are both independent of $e$; we can then denote them by $C_n(0)$ and $C_n^{\le B}(0)$.

Next, let us consider the case $n=2d$. By applying the substitution $a\mapsto \zeta a$, it's clear that $S(\alpha)$ is real. Then, by orthogonality, we have \[N_{e,2d}(0)=q^{-de-1}\sum_\alpha |S(\alpha)|^{2d}\] and \[q^{-de-1}\sum_\alpha |S(\alpha)|^2=\#{(a,b)\in P_e^2\colon a^d=b^d}\ge q^{e+1}.\]
For fixed $B$, there are finitely many subschemes of degree at most $B$, with each supporting at most $q^B$ functionals, which gives $O_{q,B}(1)$ functionals that have degree at most $B$. For sufficiently large $e$, there is a bijection between such functionals $\alpha$ and primitive functionals $\beta\in H^0(Z,\O(de\infty))^\vee$ with minimum support $Z$. In particular, we have $S(\alpha)=q^{e+1}S_Z(\beta)$, so for large $e$ we have \[q^{-de-1}\sum_{\deg \alpha\le B}|S(\alpha)|^{2d}=q^{de+2d-1}C_{2d}^{\le B}(0).\] On the other hand, the trivial bound $|S(\alpha)|\le q^{e+1}$ gives \[q^{-de-1}\sum_{\deg \alpha\le B}|S(\alpha)|^2=O_{q,B}(q^{2(e+1)-(de+1)})=o(q^{e+1}).\]
By H\"older's inequality, we obtain \[q^{-de-1}\sum_{\deg \alpha>B}|S(\alpha)|^{2d}\ge \left(q^{-de-1}\sum_{\deg \alpha>B}|S(\alpha)|^2\right)^d\ge(1-o(1))q^{d(e+1)}.\]
Then, by first taking the limit as $e\to \infty$ (noting that the $o(1)$ then becomes zero) and then taking the limit as $B\to \infty$, we obtain 
\begin{equation}\label{eq:difference}
    \liminf_{e\to \infty}(q^{-(de+2d-1)}N_{e,2d}(0)-C_{2d}(0))\ge q^{1-d}>0.
\end{equation}
Finally, we handle the case $n=2d-1$. Recall that for $\deg Z\le e+1$, the restriction map $P_e\to H^0(Z,\O_Z(e\infty))$ is surjective, so we have \[q^{-e-1}\sum_{a\in P_e}\psi(\beta((a\rvert_Z)^d))=S_Z(\beta).\] By definition, for $f=-a^d$, we have \[q^{-e-1}\sum_{a\in P_e}C_{2d-1,e}^{\le e+1}(-a^d)=C_{2d}^{\le e+1}(0),\] which after applying the triangle inequality and \eqref{eq:uniformbound} twice gives \[q^{-e-1}\sum_{a\in P_e}C_{2d-1,e}(-a^d)=C_{2d}(0)+o(1).\]
So, if $N_{e,2d-1}(f)=q^{e(d-1)+2d-2}(C_{2d-1,e}(f)+o(1))$ actually held uniformly in $f$, then summing over $a$ (setting $f=-a^d$) and using that $N_{e,2d}(0)=\sum_{a\in P_e}N_{e,2d-1}(-a^d)$ would give $q^{-(de+2d-1)}N_{e,2d}(0)=C_{2d}(0)+o(1),$ which contradicts \eqref{eq:difference}!
\end{proof}

\bibliographystyle{plain}
\bibliography{optimal_waring}

\end{document}